\documentclass[]{amsart}
\usepackage[all]{xy}
\usepackage{amscd,amsthm,amssymb,amsfonts,amsmath,euscript,mathtools,relsize}
\usepackage[margin=0.7in]{geometry}
\usepackage{url}
\usepackage{tikz-cd}
\usepackage[nodisplayskipstretch]{setspace}

\usepackage{graphicx}
\theoremstyle{plain}
\newtheorem{thm}{Theorem}[section]
\newtheorem{lemma}[thm]{Lemma}
\newtheorem{nota}[thm]{Notation}

\newtheorem{prop}[thm]{Proposition}
\newtheorem{cor}[thm]{Corollary}
\newtheorem{conj}[thm]{Conjecture}

\newtheorem{Ass}{Assumption}

\theoremstyle{definition}
\newtheorem{defn}[thm]{Definition}

\theoremstyle{remark}
\newtheorem{remark}[thm]{Remark}
\newtheorem{remarks}[thm]{Remarks}
\newtheorem*{thank}{{\bf Acknowledgments}}
\newcommand{\nc}{\newcommand}

\def\makeop#1{\expandafter\def\csname#1\endcsname
  {\mathop{\rm #1}\nolimits}\ignorespaces}
\makeop{Hom}   \makeop{End}   \makeop{Aut}   \makeop{Isom}  \makeop{Pic}
\makeop{Gal}   \makeop{ord}   \makeop{Char}  \makeop{Div}   \makeop{Lie}
\makeop{PGL}   \makeop{Corr}  \makeop{PSL}   \makeop{sgn}   \makeop{Spf}
\makeop{Spec}  \makeop{Map}    \makeop{Nm}    \makeop{Fr}    \makeop{disc}
\makeop{Proj}  \makeop{supp}  \makeop{ker}   \makeop{im}    \makeop{dom}
\makeop{Coker} \makeop{Stab}  \makeop{SO}    \makeop{SL}    \makeop{SL}
\makeop{Cl}    \makeop{cond}  \makeop{Br}    \makeop{inv}   \makeop{rank}
\makeop{id}    \makeop{Fil}   \makeop{Frac}  \makeop{GL}    \makeop{SU}
\makeop{Trd}   \makeop{Sp}    \makeop{Tr}    \makeop{Trd}   \makeop{diag}
\makeop{Res}   \makeop{ind}   \makeop{depth} \makeop{Tr}    \makeop{st}
\makeop{Ad}    \makeop{Int}   \makeop{tr}    \makeop{Sym}   \makeop{can}
\makeop{length}\makeop{SO}    \makeop{torsion} \makeop{GSp} \makeop{Ker}
\makeop{Adm}   \makeop{Frob}  \makeop{id}    \makeop{Tor}   \makeop{Ind}
\makeop{CoInd} \makeop{Inf}   \makeop{Cor}   \makeop{CoInf} \makeop{coLie}

\def\makebb#1{\expandafter\def
  \csname bb#1\endcsname{{\mathbb{#1}}}\ignorespaces}
\def\makebf#1{\expandafter\def\csname bf#1\endcsname{{\bf
      #1}}\ignorespaces}
\def\makegr#1{\expandafter\def
  \csname gr#1\endcsname{{\mathfrak{#1}}}\ignorespaces}
\def\makescr#1{\expandafter\def
  \csname scr#1\endcsname{{\EuScript{#1}}}\ignorespaces}
\def\makecal#1{\expandafter\def\csname cal#1\endcsname{{\mathcal
      #1}}\ignorespaces}

\def\doLetters#1{#1A #1B #1C #1D #1E #1F #1G #1H #1I #1J #1K #1L #1M
                 #1N #1O #1P #1Q #1R #1S #1T #1U #1V #1W #1X #1Y #1Z}
\def\doletters#1{#1a #1b #1c #1d #1e #1f #1g #1h #1i #1j #1k #1l #1m
                 #1n #1o #1p #1q #1r #1s #1t #1u #1v #1w #1x #1y #1z}
\doLetters\makebb   \doLetters\makecal  \doLetters\makebf
\doLetters\makescr
\doletters\makebf   \doLetters\makegr   \doletters\makegr

\normalsize

\makeop{Bl}

\def\Spec{{\rm Spec}\,}

\def\Qpbar{\overline{{\bbQ}_p}}
\def\Qp{{\bbQ}_p}

\def\Zp{{\bbZ}_p}
\def\Qbar{\overline{\bbQ}}

\newcommand{\Z}{\mathbb Z}
\newcommand{\Q}{\mathbb Q}

\newcommand{\C}{\mathbb C}

\renewcommand{\O}{\mathcal O} 
\newcommand{\F}{\mathbb F}

\newcommand{\p}{\mathfrak p}

\newcommand{\npr}{\noindent }

\nc{\embed}{\hookrightarrow}

\newcommand{\DR}{\{0,\infty\}_{DM}}

\nc{\ol}{\overline}
\nc{\wt}{\widetilde}
\nc{\wh}{\widehat}
\nc{\opp}{\mathrm{opp}}

\makeop{Ram}
\makeop{Rep}
\makeop{lcm}

\title[]{A refined study of Mazur's Eisenstein theory}
\author{Jun Wang}
\address{
Department of Mathematics, The University of British Columbia}
\email{junwang@math.ubc.ca}
\begin{document}

\begin{abstract}
We apply the methods of Fukaya, Kato and Sharifi to refine Mazur's study
of the Eisenstein ideal. Given prime numbers $N$ and $p\geq 5$ such that $p\mid \varphi(N)$, we study the quotient of the cohomology group of modular curve $X_{0}(N)$ by the square of the Eisenstein ideal. We study two invariants $b,c$ attached to this quotient and compute $c$. We propose a conjecture about the invariant $b$ which relates the structure of the ray class group of conductor $N$ to the modular symbols of $X_{0}(N)$. Assuming this conjecture, We compute the invariant $b$.
\end{abstract}
\maketitle
\setcounter{tocdepth}{1}
\tableofcontents

\section{Introduction}
In Mazur's seminal work \cite{Ma1}, he gave a description of the Eisenstein quotient of the \'etale cohomology group of modular curve of prime level. In this paper, we use the methods of Fukaya, Kato and Sharifi to study the quotient of this cohomology group by the square of the Eisenstein ideal.

\subsection{Motivations}
Let $N$ be a prime and let $p\geq 5$ such that $p\mid(N-1)$. let $\mathfrak{h}$ be the $p$-adic Hecke algebra acting on the cohomology of \'etale cohomology group of modular curve $X_{0}(N)$, let $I$ be the Eisenstein ideal.
In \cite{Ma1}, Mazur proved the following theorem:
\begin{thm}[Mazur]
There is a direct sum decomposition of $\Gal(\Qbar/\Q)$-modules $H/IH=H^{+}/IH^{+}\oplus H^{-}/IH^{-}$,
where $H=H^1_{\text{\'et}}(X_{0}(N)_{/\bar{\Q}},\Zp)$, $H^{\pm}$ denote the eigenspaces for complex conjugation.
\end{thm}

One natural question is: what can we say about the $\mathfrak{h}[\Gal(\Qbar/\Q)]$-module $H/I^2H$? 
Let $K=\Q(\zeta_{q})$, where $q$ is the highest power of $p$ dividing $N-1$. Viewing $H/I^2H$ as representation of $\Gal(\overline{\Q}/K)$, we have following four \textbf{group homomorphisms}:
\begin{enumerate}\label{abcd}
\item $a: \Gal(\overline{\Q}/K)\rightarrow \text{Aut}_{\mathfrak{h}}(H^-/I^2H^-)$,
\item $b: \Gal(\overline{\Q}/K)\rightarrow \Hom_{\mathfrak{h}}(H^+/I^2H^+,H^-/I^2H^-)$,
\item $c: \Gal(\overline{\Q}/K)\rightarrow \Hom_{\mathfrak{h}}(H^-/I^2H^-,H^+/I^2H^+)$,
\item $d: \Gal(\overline{\Q}/K)\rightarrow \text{Aut}_{\mathfrak{h}}(H^+/I^2H^+)$.
\end{enumerate}

One can prove that the homomorphisms $a$ and $d$ cut out a field extension $F_{0}$, the homomorphism $b$ cuts out a field $F_{-1}$, and the homomorphism $c$ cuts out a field $F_{1}$, where $F_{i}$ means that the group $\Gal(\Q(\zeta_{q})/\Q)$ acts on $\Gal(F_{i}/K)$ via $\chi_{p}^{i}$. 

Let $U=\F_{N}^{\times}/(\F_{N}^{\times})^q$. Using global class field theory and Galois cohomology, we have the following canonical isomorphisms of abelian groups:
\begin{enumerate}\label{classfield}
    \item $\Gal(F_{0}/K)\cong U$,
    \item $\Gal(F_{1}/K)\cong \mu_{q}$,
    \item $\Gal(F_{-1}/K)\cong U^{\otimes 2}\otimes \mu_{q}^{-1}$,
\end{enumerate}
where $\mu_{q}^{-1}=\Hom(\mu_{q}, \Z/q\Z)$. The $F_{i}$ are class fields, and all of them are totally tamely ramified at $N$ over $K$.
\begin{remark}
We identify $\Gal(F_{-1}/K)$ with $H^2(\Z[\tfrac{1}{p},\zeta_{N}],\Zp(2))_{G}$, where $G=Gal(\Q(\zeta_{N})/\Q)$. The isomorphism $(3)$ is from this identification, for the details, one can see Proposition \ref{Prop 4.6}.
\end{remark}
In \cite{Ma1}, Mazur uses modular symbols to construct a canonical winding isomorphism:
\begin{equation*}\label{winding}
e:I/I^2\xrightarrow{\sim}H^-/IH^-\otimes \mu_{q}\xrightarrow{\sim} U.
\end{equation*}

We know that the images of $a,d$ are both in $1+I/I^2$. Identifying $1+I/I^2$ with $I/I^2$, we may view $a$ and $d$ as homomorphisms:
\begin{enumerate}\label{ad}
\item $a:U\rightarrow I/I^2$,
\item $d:U\rightarrow I/I^2$.
\end{enumerate}

Using the winding homomorphism $e$, we get maps $I/I^2\rightarrow I/I^2$ that are multiplication by an integer in $\Z/q\Z$. We use $\widetilde{a}$ and $\widetilde{d}$ to denote these two numbers. The invariants $\widetilde{a}$ and $\widetilde{d}$ have been already computed by B. Mazur, F. Calegari, R. Sharifi and W. Stein in the unpublished manuscript \cite{Ma2}, and their results are $\widetilde{a}=-1, \widetilde{d}=1$.

For the homomorphisms $b$ and $c$, we need a more careful analysis. We carefully choose an isomorphism $H^+/IH^+\cong \Z/q\Z$: see Remark \ref{Rem 3.11}. For $H^-/IH^-$, we can identify it with $I/I^2\otimes\mu_{q}^{-1}$ using $e$.

It is easy to see that the image of $b\in\Hom_{\mathfrak{h}}(H^+/IH^+$, $IH^-/I^2H^-)$, the image of $c\in\Hom_{\mathfrak{h}}(H^-/IH^-, IH^+/I^2H^+)$.

Since 
\begin{enumerate}
    \item $\Hom_{\mathfrak{h}}(H^+/IH^+$, $IH^-/I^2H^-)\cong IH^-/I^2H^-\cong I^2/I^3\otimes\mu_{q}^{-1}$,
    \item $\Hom_{\mathfrak{h}}(H^-/IH^-, IH^+/I^2H^+)\cong \Hom_{\mathfrak{h}}(I/I^2\otimes \mu_{q}^{-1},I/I^2)\cong \mu_{q}$,
\end{enumerate}
we can rewrite $b$ and $c$ as follows:
\begin{enumerate}
    \item $b:\Gal(F_{-1}/K)\cong H^2(\Z[\tfrac{1}{p},\zeta_{N}],\Zp(2))_{G}\cong U^{\otimes 2}\otimes\mu_{q}^{-1}\cong I^2/I^3\otimes\mu_{q}^{-1}\rightarrow I^2/I^3\otimes \mu_{q}^{-1}$,
    \item $c:\Gal(F_{1}/K)\cong \mu_{q}\rightarrow \mu_{q}$.
\end{enumerate}
Hence the homomorphisms $b$ and $c$ give us two integers in $\Z/q\Z$ which were constructed in \cite{Ma2}. Let's use $\widetilde{b}$ and $\widetilde{c}$ to denote them. The integers $\widetilde{b}$ and $\widetilde{c}$ are units in $\Z/q\Z$, which can be seen using \cite{CE}, or \cite{WWE}. What we want to understand in the paper are the exact values of $\widetilde{b}$ and $\widetilde{c}$.

\begin{remarks}
\begin{enumerate}
\item The map $b$ is a map from class group $\Gal(F_{-1}/K)$ to $IH^{-}/I^2H^{-}$. Note that $IH^{-}/I^2H^{-}$ can be described explicitly by modular symbols, so the map $b$ provides an explicit presentation of class group $\Gal(F_{-1}/K)$ via modular symbols. This is one of the motivations of this paper.
\item The product $\widetilde{b}\widetilde{c}$ is related to Emmanuel Lecouturier's higher Eisenstein elements. For the details, one can see \cite[section 6.6]{Lec2}.
\end{enumerate}
\end{remarks}

\subsection{Results and the relations to the work of Fukaya, Kato and Sharifi}
Our first result is:
\begin{thm}
The invariant $\tilde{c}$ is 1.
\end{thm}

\begin{remark}
Our proof follows from the method of \cite[Section 9.6]{FK}.
\end{remark}

The homomorphism $b$ has a conjectural inverse which is closely related to the work of Fukaya, Kato and Sharifi. We give a brief introduction of their work here.

In \cite{Sha1}, for positive integers $M$ and prime $p\geq 5$ such that $p\nmid M$, Sharifi formulated a series of  remarkable conjectures which relate the arithmetic of cyclotomic fields to the homology of modular curves. Roughly speaking, he used cup products of cyclotomic units to define two maps:
\begin{equation}
\varpi_{M_{r}}^{0}: H_{1}(X_{1}(M_{r}),\tilde{C}_{\infty},\Zp)^+\rightarrow H^2(\Z[\tfrac{1}{Np},\zeta_{M_{r}}],\Zp(2))^{+},
\,\,\,\,\varpi_{M_r}: H_1(X_{1}(M_{r}),\Zp)^{+}\rightarrow H^2(\Z[\tfrac{1}{p},\zeta_{M_{r}}],\Zp(2))^{+},
\end{equation}
where $M_{r}:=Mp^r (r\geq 0)$, $\tilde{C}_{\infty}$ is the set of cusps of $X_{1}(M_{r})$ that do not lie above the $0$-cusp of $X_{0}(M_{r})$. The map $\varpi_{M_r}^0$ takes the adjusted Manin symbol $[u,v]^{*}$ (Definition \ref{Def 4.5}) to the cup product $(1-\zeta_{M_{r}}^u,1-\zeta_{M_{r}}^v)$. 

The first conjecture Sharifi made is:
\begin{conj}[Sharifi](Eisenstein quotient conjecture)\label{Con 1.1}
Let $\mathfrak{I}_{\infty}$ be the ideal of the modular Hecke algebra generated by $T_{l}-1-\langle l\rangle l\,\, \text{for}\,\,l\nmid M_{r}$, and $T_{q}-1$ for $q\mid  M_r$, and let $I_{\infty}$ be the image of $\mathfrak{I}_{\infty}$ in the cuspidal Hecke algebra. Then $\varpi_{M_r}^{0}, \varpi_{M_r}$ induce maps:
$$\varpi_{M_r}^{0}: H_{1}(X_{1}(M_{r}),\tilde{C}_{\infty},\Zp)^+/\mathfrak{I}_\infty\rightarrow H^2(\Z[\tfrac{1}{Mp},\zeta_{M_{r}}],\Zp(2))^{+},$$
$$\varpi_{M_r}: H_1(X_{1}(M_{r}),\Zp)^{+}/I_\infty \rightarrow H^2(\Z[\tfrac{1}{p},\zeta_{M_{r}}],\Zp(2))^{+}.$$
\end{conj}

\begin{remark}
In \cite{FK}, Fukaya and Kato proved these conjectures for the modular curve $X_{1}(M_{r})$ when $r\neq 0$. For modular curves with levels not divisible by $p$, the conjecture is still open.
\end{remark}

Now let's fix an even character
\[\theta: (\Z/Np\Z)^{\times}\rightarrow \Qpbar^{\times},\]
and let $\omega$ be the teichm\"{u}ller character.
Under the following assumptions:
\begin{Ass}
\begin{enumerate}
    \item $p\nmid\varphi(M)$ and $r>0$.
    \item $\theta$ is primitive.
    \item If the restriction of $\theta\omega^{-1}$ to $(\Z/p\Z)^{\times}\subset(\Z/Mp\Z)^{\times}$ is trivial, then the restriction $\theta\omega^{-1}|_{(\Z/M\Z)^{\times}}$ satisfies $\theta\omega^{-1}(p)\neq 1$.
    \item In the case $M=1$, $\theta\neq \omega^2$.
\end{enumerate}
\end{Ass}
Sharifi constructed a map
\[\Upsilon_{\theta}: H^2(\Z[\tfrac{1}{p},\zeta_{M_{r}}],\Zp(2))_{\theta}\rightarrow H_1(X_{1}(M_{r}),\Zp)_{\theta}/I_{\infty}\]
and he made another conjecture:
\begin{conj}\label{Sharificonj}
The map $\varpi_{\theta}$ and $\Upsilon_{\theta}$ are inverse to each other.
\end{conj}

\begin{remark}
This conjecture has been proved by Fukaya and Kato under certain assumptions. For the details, one can see \cite{FK}.
\end{remark}

In our setting, we propose a new conjecture that is analogous to Conjecture \ref{Sharificonj}.
\begin{conj}\label{ourconjecture}
The following two maps are inverse to each other.
\[b: H^2(\Z[\tfrac{1}{p},\zeta_{N}],\Zp(2))_{G}\rightarrow H_1(X_{1}(N),\Zp)_{\Delta}/I_{\infty},\]
\[\varpi_{N,G}:H_1(X_{1}(N),\Zp)_{\Delta}/I_{\infty}\rightarrow H^2(\Z[\tfrac{1}{p},\zeta_{N}],\Zp(2))_{G}.\]
Here $\Delta$ is the automorphism group generated by diamond operators.
\end{conj}

\begin{remarks}
Our setting is different from the setting in \cite{Sha1} in the following ways:
\begin{enumerate}
    \item The level $N$ is not divisible by $p$. Moreover, $N\equiv 1\bmod p$.
    \item Our tame character is the trivial character since we are essentially considering the modular curve $X_{0}(N)$.
    \item In the work of Fukaya, Kato and Sharifi, their $H^2(\Z[\tfrac{1}{p},\zeta_{M_{r}}],\Zp(2))_{\theta}$ is related to unramified extension of cyclotomic field. In our case, $H^2(\Z[\tfrac{1}{p},\zeta_{N}],\Zp(2))_{G}$ is related to the tame ramified extension $F_{-1}/K$.
\end{enumerate}
\end{remarks}

Using Mazur's results and duality, we can relates the homology of $X_1(N)$ with the cohomology of $X_{0}(N)$. Our result for $\widetilde{b}$ is:
\begin{thm}
Assume Conjecture \ref{Con 1.1} for $X_1(N)$ and Conjecture \ref{ourconjecture}, the invariant $\widetilde{b}$ is $1$.
\end{thm}

In Sections 2 and 3, we review Mazur's work on $X_{0}(N)$. In Section 4, we define the extension classes that we are interested in. In Section 5 and 6, we investigate these classes carefully.

\begin{nota}
Let $N\geq 5$ and $p\geq 5$ be prime numbers such that $N\equiv 1\bmod p$. We fix an embedding $\overline{\Q}\rightarrow\C$ and an embedding $\overline{\Q}\rightarrow\Qpbar$.

\npr For $n\geq 1$, let $\zeta_{n}=e^{\frac{2\pi i}{n}}\in\Qbar\subset\C$ and $\zeta_{n}^+=\zeta_{n}+\zeta_{n}^{-1}$.
\end{nota}

\begin{thank}
The results of this paper are a part of the author's Ph.D. thesis in University of
Arizona. The author would like to thank his advisor, Prof. Romyar Sharifi, for his guidance, support, and
suggesting this problem. The author was partially supported by National Science Foundation under Grants
No. DMS-1360583 and No. DMS-1401122.
Also, the
author would like to thank Emmanuel Lecouturier, Preston Wake and Sujatha Ramdorai for helpful
suggestions during preparation of this article.
\end{thank}

\section{Mazur's results on the homology of $X_{0}(N)$}
For general theory of modular curves, see \cite{DI}, \cite{DR}, \cite{G}, \cite{KM}, \cite{Oh2}.
\begin{defn}\label{Def 2.1}
Let $\mathcal{Y}_{1}(N)$ be the $\Z[1/N]$-scheme that represents the functor taking a $\Z[1/N]$-scheme $S$ to the set of pairs $(E,\alpha)$, where $E$ is a elliptic curve over $S$ and $\alpha$ is a closed immersion $\Z/N\Z\rightarrow E$ of $S$-group schemes. Let $Y_1(N):=\mathcal{Y}_{1}(N)\otimes{\Q}$. 
Let $\mathcal{X}_{1}(N)$ be the $\Z[1/N]$-scheme representing the functor taking a $\Z[1/N]$-scheme $S$ to the set of pairs $(E,\alpha)$, where $E$ is a generalized elliptic curve over $S$ and $\alpha$ is a closed immersion $\Z/N\Z\rightarrow E^{reg}$ of $S$-group schemes. Here $E^{reg}$ denotes the smooth locus of $E/S$. Let $X_{1}(N):=\mathcal{X}_{1}(N)\otimes \Q$.
\end{defn}

\begin{defn}\label{Def 2.4}
Let $\mathcal{Y}_{0}(N)$ be the coarse moduli scheme over $\Z[\tfrac{1}{N}]$ classifying the pairs $(E,C)$, where $E$ is a elliptic curve over a $\Z[\tfrac{1}{N}]$-scheme $S$ and $C\subset E[N]$ is a finite flat group scheme of order $N$ that is locally free and of rank $N$. Let $Y_{0}(N):=\mathcal{Y}_{0}(N)\otimes{\Q}$. 
Let $\mathcal{X}_{0}(N)$ be the coarse moduli scheme over $\Z[\tfrac{1}{N}]$ classifying generalized elliptic
curves together with a locally free subgroup scheme of order $N$ in place of $\alpha$ in Definition \ref{Def 2.1}. Let $X_{0}(N):=\mathcal{X}_{0}(N)\otimes{\Q}$.
\end{defn}

Let $\mathfrak{h}_{0}(N)$ be the subring of $\End_{\Zp}(H_1(X_{0}(N)(\C)\Zp))$ generated by $T(n)\,(n\geq 1, n\nmid N)$ and the Atkin-Lehner involution $W_{N}$. Let $I$ be the ideal of $\mathfrak{h}_{0}(N)$ that is generated by $\{T(l)-1-l,W_{N}+1\}$ for all prime $l\neq N$.

Let $H_1:=H_{1}(X_{0}(N)(\C),\Zp)$. Let $H_1^{\pm}$ be the eigensubspaces of complex conjugation.

\begin{defn}\label{Def 3.2}
Let 
\[e^{+}:I\rightarrow H_1^+\]
be the $\mathfrak{h}_{0}(N)$-morphism such that
\[x\mapsto x\{0,\infty\}_{DM},
\]
where $\{0,\infty\}_{DM}$ is the image of $\{0,\infty\}$ under the Drinfeld-Manin splitting map:
\[H_1(X_{0}(N)(\C),\textnormal{cusps},\Qp)\rightarrow H_1(X_{0}(N)(\C),\Qp).\]
We call this map the winding homomorphism.
\end{defn}

\begin{remark}\label{Rem 3.4}
These maps are well defined since
\[
IH_{1}(Y_{0}(N)(\C),\textnormal{cusps},\Zp)\subset H_{1}(X_{0}(N)(\C),\Zp).
\]
\end{remark}

\begin{defn}[Shimura subgroup]\label{Def 3.3}
The Shimura covering is the maximal unramified subcover intermediate to $X_{1}(N)\rightarrow X_{0}(N)$  whose covering group is the unique quotient group of $(\Z/N\Z)^{\times}/\{\pm 1\}$. From the natural map $X_{1}(N)\rightarrow X_{0}(N)$, we obtain a Picard morphism $\textnormal{Jac}(X_{0}(N))\rightarrow \textnormal{Jac}(X_{1}(N))$. We define the Shimura subgroup $\Sigma$ to be the kernel of this map on $\Qbar$-points.
\end{defn}

In \cite{Ma2}, Mazur proved the following theorem about the structure of the Hecke algebra and the winding homomorphism.
\begin{thm}[Mazur]\label{mazur}
\begin{enumerate}
	\item There is an isomorphism:
	\[\mathfrak{h}_{0}(N)/I\rightarrow \Zp/(\xi),
	\]
	\[T(l)\mapsto 1+l,
	\]
	where $\xi=\tfrac{N-1}{12}$.
	\item There is a canonical isomorphism \[\phi:(\Z/N\Z)^{\times}\otimes\Zp\rightarrow H_1^{+}/IH_1^{+},\]which identifies $H_1^{+}/IH_1^{+}$ with the Galois group of $q$-subcover of the Shimura covering.
	\item Let $a$ and $b$ be coprime integers with $b$ relatively prime to $N$. Let $\bar{b}$ denote the image of $b$ in $(\Z/N\Z)^{\times}$. Let $\Phi(a/b)\in H^+/IH^+$ denote the image of the modular symbol $\{0,a/b\}$ in $H_1^+/IH_1^+$. Then
	\[\Phi(a/b)=\phi({\bar{b}}^{-1}).\]
	\item Let $\mathfrak{P}$ be the ideal $(p,I)$ of $\mathfrak{h}_{0}(N)$, the winding homomorphism $e^+: I_{\mathfrak{P}}\rightarrow H_{1,\mathfrak{P}}^+$ is an isomorphism of Hecke modules.
	\item Let $\eta_{l}=1+l-T(l)$ and $\Tilde{e}^+:I/I^2\rightarrow H_1^+/IH_1^+$ be the homomorphism induced from $e^+$. Then:
	\[
	\Tilde{e}^+(\eta_{l})=(l-1)\phi(\bar{l}),
	\]
	where $l$ is any prime number different from $N$.
	\item The Hecke algebra $\mathfrak{h}_{0}(N)_{\mathfrak{P}}$ is Gorenstein.
\end{enumerate}
\end{thm}

\section{Structure of cohomology groups modulo Eisenstein ideal}
In this section, we state the results about the structure of cohomology groups modulo Eisenstein ideal, the proofs are just in the same way as in \cite[Section 6]{FK}.

\begin{nota}\label{Not 3.3}
Let $H^1(X)=H_{\textnormal{\'et}}^1(X_{0}(N))$ and $H^1(Y)=H_{\textnormal{\'et}}^1(Y_{0}(N))$.
\end{nota}

\begin{remark}
We have an isomorphism from Poincar\'e duality
\[H^1(X)(1)\cong H_1(X_{0}(N),\Zp),\]
which respects complex conjugation. We use these isomorphisms to transfer from homology groups to cohomology groups.
\end{remark}
 
Via duality, we have the following exact sequence of $\mathfrak{h}_{0}(N)[\Gal(\overline{\Q}/\Q)]$-modules:
\begin{equation}\label{Ses 3.1}
0\rightarrow H^1(X)(1)\rightarrow H^1(Y)(1)\rightarrow \Zp \rightarrow 0,\,\,\,\,
\end{equation}




\begin{lemma}\label{Lem 3.2}
We have
\[
H^1(Y)_{DM,\mathfrak{P}}/H^1(X)_{\mathfrak{P}}\cong \mathfrak{h}_{0}(N)/I
\]
\end{lemma}

\begin{proof}
It is easy to see that $H^1(Y)_{DM,\mathfrak{P}}/H^1(X)_{\mathfrak{P}}$ is generated by the image of $\{0,\infty\}_{DM,\mathfrak{P}}$. The isomorphism
\[
H^1(Y)_{DM,\mathfrak{P}}/H^1(X)_{\mathfrak{P}}\cong \mathfrak{h}_{0}(N)/I
\]
follows from the fact that $(0)-(\infty)$ has order $\xi$.
\end{proof}

\begin{defn}\label{Def 3.7}
We define a $\Zp$-linear map from $H^1(X)/IH^1(X)$ to $\Zp/(\xi)$ as follows.
Since $\xi\{0,\infty\}_{DM,\mathfrak{P}}\in H^1(X)_{\mathfrak{P}}$, we get a map $f$:
\[H^1(X)_{\mathfrak{P}}\xrightarrow{f} \Zp/\xi\Zp\]
\[x\mapsto (x,W_{N}\xi\{0,\infty\}_{DM,\mathfrak{P}}),\]
where $(,)$ is the cup product of \'etale cohomology group.
Since $\mathfrak{h}_{0}(N)/I\cong \Zp/\xi\Zp$, the above map factors through $H^1(X)/IH^1(X)$.
\end{defn}

\begin{thm}[Fukaya, Kato]\label{FKsplit}
We have an exact sequence of $\mathfrak{h}_{0}(N)[\Gal(\overline{\Q}/\Q)]$-modules:
\begin{equation}\label{Ses 3.2}
    0\rightarrow H^1(X)^{-}/IH^1(X)^{-}\rightarrow H^1(X)/IH^{1}(X)\xrightarrow{f} \Zp/\xi\rightarrow 0
\end{equation}
satisfying the following condition:
\begin{enumerate}
	\item The action of $\Gal(\Qbar/\Q)$ on $ H^1(X)^{-}/IH^1(X)^{-}$ is given by $\kappa^{-1}$, where $\kappa$ is the cyclotomic character.
	\item The action of $\Gal(\Qbar/\Q)$ on $ \Zp/\xi$ is trivial.
\end{enumerate}
\end{thm}
\begin{proof}
The proof can be found in \cite[Section 6]{FK}.
\end{proof}

\begin{remark}\label{Rem 3.11}
By Theorem \ref{FKsplit} and multiplicity one theorem, we have an isomorphism: 
$$H^1(X)^{+}/IH^{1}(X)^{+}\cong \Zp/\xi\Zp.$$
we call the element of $H^1(X)^{+}/IH^1(X)^{+}$ corresponding to $1\in\Zp/\xi\Zp$ the \em{canonical generator}.
\end{remark}

In fact, Mazur proved a stronger result.
\begin{thm}[Mazur]\label{Thm 3.3}
The exact sequence
\begin{equation}\label{Globalses}
 0\rightarrow H^1(X)^{-}/IH^1(X)^{-}\rightarrow H^1(X)/IH^{1}(X)\xrightarrow{f} H^1(X)^{+}/IH^1(X)^{+}\rightarrow 0,
\end{equation}
of Galois modules splits. The map 
\[H^1(X)^{+}/IH^1(X)^{+}\rightarrow \Coker(H^1(X_{1}(N))\rightarrow H^1(X))\]
is an isomorphism, which gives a splitting of (\ref{Globalses}).
\end{thm}

\begin{proof}
See \cite[Section 16]{Ma1}.
\end{proof}

\begin{remark}\label{Rem 3.12}
In fact, Mazur proved that there is a decomposition of the $I$-torsion on the Jacobian:
\[\textnormal{Jac}(X_{0}(N))_{\mathfrak{P}}(I)=\Sigma_{p}\oplus C_{p},\]
where $\Sigma_{p}$ is the $p$-part of the Shimura subgroup $\Sigma$ and $C_{p}$ is the $p$-part of the cuspidal subgroup. Via the duality between the \'etale cohomology group and the Picard group, we have
\[
H^1(X)^{-}/IH^1(X)^{-}\cong \Hom(\Sigma,\Qp/\Zp), 
\]
\[
H^1(X)^{+}/IH^1(X)^{+}\cong \Hom(C,\Qp/\Zp),
\]
\end{remark}

\section{Extension classes}
In this section, we construct the extension classes we are interested in and compute the corresponding invariants. The construction of such extension classes has been considered in \cite{Ma2}; we revisit it in this section.

Let $q=p^f$ be the highest $p$ power dividing $N-1$, and let $K=\Q(\zeta_{q})$. Let $U=\F_{N}^{\times}/(\F_{N}^{\times})^{q}$.

We set up the following notation.
\begin{nota}\label{Not 3.4}
	Let
	\begin{equation*}
	P=H^1(X)^{-}/IH^1(X)^{-},\,Q=H^1(X)^{+}/IH^1(X)^{+},
	\end{equation*}
\end{nota}

We want to consider the following exact sequence:
\begin{equation}
    0\rightarrow H^1(X)/IH^1(X)\otimes I/I^2\rightarrow H^1(X)/I^2H^1(X)\rightarrow H^1(X)/IH^1(X)\rightarrow 0. 
\end{equation}

By pushout and pullback, we get the following exact sequences:
\begin{enumerate}
    \item[(A)]$0\rightarrow Q\otimes I/I^2\rightarrow ?\rightarrow Q\rightarrow 0$,
    \item[(B)]$0\rightarrow P\otimes I/I^2\rightarrow ?\rightarrow Q\rightarrow 0$,
    \item[(C)]$0\rightarrow Q\otimes I/I^2\rightarrow ?\rightarrow P\rightarrow 0$,
    \item[(D)]$0\rightarrow P\otimes I/I^2\rightarrow ?\rightarrow P\rightarrow 0$.
\end{enumerate}

\begin{nota}\label{Not 4.1}
Let $a\in H^1(\Z[\frac{1}{Np}],I/I^2)$ correspond to the exact sequence (A), $b\in H^1(\Z[\frac{1}{Np}],P\otimes I/I^2)$ correspond to the exact sequence (B), $c\in H^1(\Z[\frac{1}{Np}],P^{-1}\otimes I/I^2)$ correspond to the exact sequence (C), and $d\in H^1(\Z[\frac{1}{Np}],I/I^2)$ correspond to the exact sequence (D).  
\end{nota}

\begin{defn}\label{Def 4.1}
Let
\begin{enumerate}
\item $\chi_{a}\in H^1(\Z[1/Np,\zeta_{q}], I/I^2)^{\Gal(K/\Q)}$ be the image of $a$,
\item $\chi_{b}\in H^1(\Z[1/Np,\zeta_{q}],P\otimes I/I^2)^{\Gal(K/\Q)}$ be the image of $b$,
\item $\chi_{c}\in H^1(\Z[1/Np,\zeta_{q}], P^{-1}\otimes I/I^2)^{\Gal(K/\Q)}$ be the image of $c$,
\item $\chi_{d}\in H^1(\Z[1/Np,\zeta_{q}], I/I^2)^{\Gal(K/\Q)}$ be the image of $d$.
\end{enumerate}
\end{defn}

\begin{prop}\label{Prop 4.1}
For $k=1,-1,0$, there exists a unique field extension $F_{k}/K$ that is Galois over $\Q$, satisfying the following conditions:
\begin{enumerate}
    \item $\Gal(K/\Q)$ acts on $\Gal(F_{k}/K)$ via $\kappa^{k}$, where $\kappa$ is the mod $q$ cyclotomic character,
    \item $\Gal(F_{k}/K)\cong \Z/q\Z$,
    \item $F_{k}/K$ is unramified outside $p$ and $N$,
    \item $F_{k}/K$ is totally tamely ramified at $N$,
    \item $F_{0}$ is the unique $q$-subextension of $K(\zeta_{N})/K$,
    \item $F_{1}$ is peu ramifi\'ee at $p$ over $K$ and equals $K(N^{1/q})$,
    \item $F_{-1}/K$ splits completely at $p$.
\end{enumerate}
\end{prop}

\begin{proof}
See \cite[Lemma 3.9, Proposition 5.4]{CE}.
\end{proof}

\begin{nota}\label{Not 4.2}
We use $G_{k}$ to denote $\Gal(F_{k}/K)$.
\end{nota}

By Proposition \ref{Prop 4.1}, we know that $\chi_{a},\chi_{d}\in \Hom(G_{0},I/I^2)$, $\chi_{b}\in\Hom(G_{-1},P\otimes I/I^2)$, $\chi_{c}\in\Hom(G_{1},P^{-1}\otimes I/I^2)$.

\begin{lemma}\label{Lem 4.1}
There are canonical isomorphisms:
\begin{enumerate}
    \item $G_{0}\cong U$,
    \item $G_{1}\cong \mu_{q}$,
    \item $G_{-1}\cong U^{\otimes 2}\otimes \mu_{q}^{-1}$.
\end{enumerate}
\end{lemma}

\begin{proof}
\begin{enumerate}
\item We have $\Gal(K(\zeta_{N})/K)\cong \F_{N}^{\times}$, so $G_{0}\cong \F_{N}^{\times}/(\F_{N}^{\times})^{q}=U$.

\item For $G_{1}$, we have the Kummer map
\[G_{1}\rightarrow \mu_{q},\,\,\,\,\sigma\mapsto \frac{\sigma(N^{\frac{1}{q}})}{N^{\frac{1}{q}}},
\]
which gives the canonical isomorphism $G_{1}\cong \mu_{q}$.

\item 
For the isomorphism
\[G_{-1}\cong U^{\otimes 2}\otimes\mu_{q}^{-1},\]
see Remark \ref{canonicaliso}.
\end{enumerate}
\end{proof}

\begin{defn}\label{Def 4.2}
By the winding isomorphism, $I/I^2$ is identified with $U$. By (5) of Theorem \ref{mazur} and Poincar\'e duality, $P$ is identified with $I/I^2\otimes\mu_{q}^{-1}$. Hence by Lemma (\ref{Lem 4.1}), $\chi_{a}, \chi_{d}$ are characters from $U$ to $U$, $\chi_{b}$ is a character from
$U^{\otimes 2}\otimes \mu_{q}^{-1}$ to $U^{\otimes 2}\otimes \mu_{q}^{-1}$, $\chi_{c}$ is a character from $\mu_{q}$ to $\mu_{q}$. So, we have four integers attached to $\chi_{a},\chi_{b},\chi_{c},\chi_{d}$, and we use $\tilde{a},\tilde{b},\tilde{c},\tilde{d}$ to denote the invariants.
\end{defn}

The invariants $\tilde{a}$ and $\tilde{d}$ have been already computed in \cite{Ma2}. We review their argument here.

Fixing a basis $e_{+},e_{-}$ of $H^1(X)_{\mathfrak{P}}^{+}/I^2\oplus H^1(X)_{\mathfrak{P}}^{-}/I^2$, one can write the representation as follows:
$$\rho(\sigma)=
\begin{pmatrix}
1+A(\sigma)& B(\sigma)\\
C(\sigma)&\kappa^{-1}(\sigma)(1+D(\sigma))
\end{pmatrix}$$
for $\sigma\in\Gal(\overline{\Q}/\Q)$, where $A,B,C,D:\Gal(\Qbar/\Q)\rightarrow I/I^2$ are homomorphisms, and $\kappa$ is the $p$-adic cyclotomic character.


\begin{prop}\label{Prop 4.2}
For a prime $l\not\equiv1\bmod q,\,\,l\nmid N$, we have
\begin{equation*}
1+A(\Frob_{l}^{-1})=1+\frac{\eta_{l}}{l-1},
\end{equation*}
where $\eta_{l}=1+l-T(l)$.
\end{prop}

\begin{proof}
From the representation $\rho$, we know that $1+A(\Frob_{l}^{-1})$ is the root of $t^2-T(l)t+l=0 \mod I^2$ which equal to $1$ modulo $I$. We claim that
\[1+A(\Frob_{l}^{-1})\equiv 1+\frac{\eta_{l}}{l-1}\bmod I^2.\]
This is seen via the following string of equalities:
\begin{eqnarray*}
\bigg(1+\frac{\eta_{l}}{l-1}\bigg)^2-(1+l-\eta_{l})\bigg(1+\frac{\eta_{l}}{l-1}\bigg)+l&=&\bigg(1+\frac{\eta_{l}}{l-1}\bigg)\bigg(\frac{\eta_{l}}{l-1}+\eta_{l}-l\bigg)+l\\
&=&\bigg(1+\frac{\eta_{l}}{l-1}\bigg)\bigg(\frac{l}{l-1}\eta_{l}-l\bigg)+l\\
&=&l\bigg(\frac{\eta_{l}}{l-1}+1\bigg)\bigg(\frac{\eta_{l}}{l-1}-1\bigg)+l\\
&=&l\bigg(\bigg(\frac{\eta_{l}}{l-1}\bigg)^2-1\bigg)+l\\
&\equiv& 0\mod I^2
\end{eqnarray*}
\end{proof}

\begin{thm}[Calegari-Mazur-Sharifi-Stein]\label{Prop 4.3}
We have $\tilde{a}=-\tilde{d}=-1$.
\end{thm}

\begin{proof}
Since 
\[1+A(\Frob_{l}^{-1})+\kappa^{-1}(\Frob_{l}^{-1})(1+D(\Frob_{l}^{-1}))=T(l)=1+l-\eta_{l},\]
\npr and
\[1+A(\Frob_{l}^{-1})\equiv 1+\frac{\eta_{l}}{l-1}\mod I^2,\]
\npr we know that
\[1+D(\Frob_{l}^{-1})= 1-\frac{\eta_{l}}{l-1}.\]
The image of $\Frob_{l}\in U$ is $l$, so $\chi_{a}$ maps $l^{-1}\in U$ to $\frac{1}{l-1}\eta_{l}\in I/I^2$. Via the winding isomorphism, $\eta_{l}$ maps to $l^{l-1}\in U$. Thus, the composite maps $l^{-1}$ to $l$, which means that $\tilde{a}=-1$. 

As for $\chi_{d}$, it maps $l^{-1}\in U$ to $-\frac{1}{l-1}\eta_{l}\in I/I^2$. Via the winding isomorphism, $-\frac{1}{l-1}\eta_{l}$ maps to $l^{-1}\in U$. The composite then maps $l^{-1}$ to $l^{-1}$, which means that $\tilde{d}=1$.

\end{proof}

\section{The invariant $\tilde{c}$}
In this section, we use the method developed by Fukaya and Kato in \cite[Section 9.6.3]{FK} to compute the invariant $\tilde{c}$.

Recall that the invariant $\tilde{c}$ (resp. cocycle $c$) is from the following extension class:
\begin{equation}\label{cuspidal1}
0\rightarrow \frac{I}{I^2}\otimes \frac{H^1(X)^{+}}{IH^1(X)^{+}}\rightarrow \frac{IH^1(X)^{+}}{I^2H^{1}(X)^{+}}\oplus \frac{H^{1}(X)^{-}}{IH^1(X)^{-}}\rightarrow \frac{H^{1}(X)^{-}}{IH^1(X)^{-}}\rightarrow 0
\end{equation}
In this section, we will relate the sequence (\ref{cuspidal1}) to another sequence that arises from the cupsidal extension.

We are interested in the following exact sequence:
\begin{equation}\label{cuspidal3}
    0\rightarrow \frac{H^1(Y)_{DM}^{+}}{IH^1(Y)_{DM}^{+}}\rightarrow \frac{H^{1}(Y)_{DM}} {IH^1(Y)_{DM}}\rightarrow \frac{H^{1}(Y)_{DM}^{-}} {IH^1(Y)_{DM}^{-}}\rightarrow 0.
\end{equation}

Since there is a canonical generator $\DR$ of $\frac{H^{1}(Y)_{DM}^{-}} {IH^1(Y)_{DM}^{-}}$ and $H^1(Y)_{DM,\mathfrak{P}}^+=H^1(X)_{DM,\mathfrak{P}}^{+}$, the sequence (\ref{cuspidal3}) gives an element $c'$ in $H^1(\Z[1/Np],Q(1))\cong H^1(\Z[1/Np],\mu_{q})$.

\begin{prop}\label{Prop 4.4}
We have $c=c'$ in $H^1(\Z[1/Np],\mu_{q})$.
\end{prop}

\begin{proof}
Since $IH^{1}(Y)_{DM,\mathfrak{P}}^{-}=H(X)^{-}_{\mathfrak{P}}$,
we can rewrite sequence (\ref{cuspidal1}) as follows:
$$0\rightarrow \frac{I}{I^2}\otimes \frac{H^1(X)^{+}}{IH^1(X)^{+}}\rightarrow  \frac{IH^{1}(X)^{+}}{I^2H^1(X)^{+}}\oplus\frac{IH^{1}(Y)_{DM}^{-}} {I^2H^1(Y)_{DM}^{-}}\rightarrow \frac{IH^{1}(Y)_{DM}^{-}} {I^2H^1(Y)_{DM}^{-}}\rightarrow 0.$$
\npr Since $I$ is a principal ideal in the Eisenstein component and $H^1(Y)_{DM,\mathfrak{P}}^{-}/H^1(X)_{\mathfrak{P}}^{-}\cong h/I\cong I/I^2$, choosing a generator $\eta\in I/I^2$, we get the isomorphisms $\frac{H^1(X)^{+}}{IH^1(X)^{+}}\cong \frac{I}{I^2}\otimes \frac{H^1(X)^{+}}{IH^1(X)^{+}}$ and $\frac{H^{1}(Y)_{DM}^{-}} {IH^1(Y)_{DM}^{-}}\cong \frac{IH^{1}(Y)_{DM}^{-}} {I^2H^1(Y)_{DM}^{-}}$. By doing this, we identify two exact sequences. Although the identification depends on the choice of the generator, the extension class does not change.
\end{proof}

So, in order to compute $c$ or $\tilde{c}$, it suffices to compute the extension class obtained from (\ref{cuspidal3}). We need to make some preparations. 

Let $J$ be the Jacobian variety of $X_{0}(N)$, and let $GJ$ be the generalized Jacobian variety of $X_{0}(N)$ with respect to the cusps of $X_{0}(N)$. For the general properties of Jacobians and generalized Jacobians of algebraic curves, one can see \cite[Section 3]{Oh0}.
We have the following exact sequence:
\begin{equation}\label{GJexact}
0\rightarrow T\rightarrow GJ\rightarrow J\rightarrow 0,
\end{equation}
where $T:=\Ker(GJ\rightarrow J)$, and we have a canonical isomorphism $T\cong \Coker(\mathbb{G}_{m}\rightarrow \mathbb{G}_{m}\times \mathbb{G}_{m})$. 

Note that we have the following duality between \'etale cohomology groups and generalized Jacobians
\[H^1(Y)\cong GJ[p^{\infty}]^{\wedge},\,\,H^1(X)\cong J[p^{\infty}]^{\wedge}.\]
where $\wedge$ is $\Qp/\Zp$-dual. By this duality, the sequence 
\[0\rightarrow T\rightarrow GJ\rightarrow J\rightarrow 0\]
is dual to
\[0\rightarrow H^1(X)\rightarrow H^1(Y)\rightarrow\Zp \rightarrow 0.\]

Taking the $q$-kernel of the first of the above two sequences, we have the following exact sequence:
\begin{equation}\label{Generaljacobian1}
0\rightarrow T[q]\rightarrow GJ[q]\rightarrow J[q]\rightarrow 0
\end{equation}
The exactness of above sequence is from the exactness of the following sequence:
\begin{equation}\label{torsionses}
0\rightarrow H^1(X)/q\rightarrow H^1(Y)/q\rightarrow \Zp e/q\rightarrow 0,
\end{equation}
where $e=(\infty)-(0)\in \frac{H^1(Y)}{H^1(X)}$.

\begin{nota}
Let $\xi=\frac{n}{m}$ and $n=vq$, where $(n,m)=1$. Note that $\frac{\xi}{q}$ and $\frac{1}{v}$ can be viewed as elements in $(\Z/q\Z)^{\times}$. To avoid ambiguity, we choose $\widetilde{r}, \widetilde{v}\in\Z$ such that $\widetilde{r}\equiv\frac{\xi}{q}\in\Z/q\Z, \widetilde{v}\equiv\frac{1}{v}\in\Z/q\Z$.
\end{nota}

\npr 
Let $D=v(0)-v(\infty)\in J[p^{\infty}]$. 
The following lemma is from \cite[section 9.6]{FK}.

\begin{lemma}\label{Lem 4.3}
Let \[\frac{H^1(Y)(1)_{DM}}{H^1(X)(1)}\xrightarrow{\iota} H^1(X)(1)\otimes \frac{\Qp}{\Zp}\] be the map induced by $H^1(Y)_{DM}(1)\rightarrow H^1(X)(1)\otimes_{\Zp}\frac{\Qp}{\Zp}$. Then $\iota$ maps the image of $v\{0,\infty\}_{DM}$ to the image of $D$ in $H^1(X)(1)\otimes\frac{\Qp}{\Zp}$.
\end{lemma}

\begin{proof}
The proof can be found in \cite[Lemma 9.6.5]{FK}. For the convenience of the reader, we review the proof here. Viewing $v\{0,\infty\}$ as an element in $H^1(Y)(1)$, we have $qv\{0,\infty\}=\Div(h)$ in $H^1(Y)(1)/H^1(X)(1)$ for some $h\in \O(Y)^{\times}$. This means that $-qv\{0,\infty\}\equiv [h]\bmod H^1(X)(1)$, where $[h]$ is the Kummer class of $h$. We can write
\begin{equation}\label{SBSW}
-qv\{0,\infty\}=[h]+x
\end{equation}
for some $x\in H^1(X)(1)$. We write $x$ as a projective system $(x_1,x_{2},\ldots)$, where $x_{i}\in \frac{H^1(X)(1)}{p^{i}H^1(X)(1)}$ for $i\in\Z_{>0}$.
Modulo $q$, we have \[
[h]=-x_{f} \in \frac{H^1(Y)(1)}{qH^1(Y)(1)}.\]
From the exact sequence
\[0\rightarrow H^1(X)(1)\rightarrow H^1(Y)(1)\rightarrow \Zp\rightarrow 0,\]
we know that the map $\frac{H^1(X)(1)}{qH^1(X)(1)}\xrightarrow{i} \frac{H^1(Y)(1)}{qH^1(Y)(1)}$ is injective. By \cite[Lemma 9.6.4]{FK}, we know that $i(D)=[h]$ and $i(D+x_{f})=x_{f}+[h]=0$. Hence,
\begin{equation}\label{SBSBSB}
D=-x_{f} \in \frac{H^1(X)(1)}{qH^1(X)(1)}.
\end{equation}
Taking the Drinfeld-Manin splitting map for (\ref{SBSW}), since $[h]$ maps to $0$ and $x$ maps to $x$, we have
\begin{equation}\label{SBSB}
qv\{0,\infty\}_{DM}=-x \in H^1(Y)(1)_{DM}.
\end{equation}
Note that we have a canonical injection
\begin{equation*}
\frac{H^1(Y)(1)_{DM}}{H^1(X)(1)}\xrightarrow{\iota} H^1(X)(1)\otimes \frac{\Qp}{\Zp}.
\end{equation*}
By (\ref{SBSB}), we have $\iota(v\{0,\infty\}_{DM})=(-x)\otimes\frac{1}{q}$. By (\ref{SBSBSB}), we have $(-x)\otimes\frac{1}{q}=D\in H^1(X)(1)\otimes \frac{\Qp}{\Zp}$. 
\end{proof}

\begin{remark}\label{woshisb}
The map $\iota$ maps $\{0,\infty\}_{DM}$ to $\widetilde{v}D\in J[q]$.
\end{remark}
\begin{lemma}\label{Lem 4.4}
The map
\begin{equation*}
(-,W_{N}\xi\{0,\infty\}_{DM}):H^1(X)\rightarrow \Zp/\xi\Zp\rightarrow\frac{\frac{1}{\xi \Zp}}{\Zp}\hookrightarrow\frac{\Qp}{\Zp}
\end{equation*}
corresponds to the element $-\widetilde{v}D$ in $J[p^{\infty}]$. 
\end{lemma}

\begin{proof}
We identify $\frac{H^1(Y)(1)_{DM}}{H^1(X)(1)}$ with $J[p^{\infty}]$ using the following canonical isomorphism
\[\frac{H^1(Y)(1)_{DM}}{H^1(X)(1)}\xrightarrow{\iota} H^1(X)(1)\otimes \frac{\Qp}{\Zp}=J[p^{\infty}].\]
From the pairing 
\[H^1(X)\times J[p^{\infty}]\rightarrow \frac{\Qp}{\Zp},\]
we have the following pairing
\[H^1(X)\times \frac{H^1(Y)(1)_{DM}}{H^1(X)(1)}\rightarrow \frac{\Qp}{\Zp}.\]
Via the isomorphism
\[\frac{H^1(Y)(1)_{DM}}{H^1(X)(1)}\rightarrow J[p^{\infty}],\]
$\{0,\infty\}_{DM}$ maps to $\widetilde{v}D\in J[q]$ (Remark \ref{woshisb}). Note that the image of $W_{N}$ in $\mathfrak{h}_{0}(N)_{\mathfrak{P}}$ is $-1$. Therefore, the map $(-,W_{N}\xi\{0,\infty\}_{DM})$ can be restated as
\[H^1(X)\xrightarrow{s}\frac{\Qp}{\Zp}[q]\cong \frac{\frac{1}{\xi\Zp}}{\Zp},\]
where the map $s$ corresponds to $-\widetilde{v}D\in \Hom(H^1(X),\frac{\Qp}{\Zp}[q])$. Then the lemma follows.
\end{proof}

\begin{cor}
Taking the $\frac{\Qp}{\Zp}$-dual of the map $H^1(X)\rightarrow Q=\Z/q\Z$, we get a map
\[Q^{\wedge}\rightarrow J[q].\]
Then the canonical generator of $Q^{\wedge}$ maps to
$-r\widetilde{v}D\in J[q]$.
\end{cor}

\begin{proof}
For $\tfrac{\Qp}{\Zp}[q]$, we have a canonical generator $\tfrac{1}{q}$. So we have a canonical isomorphism
\[Q^{\wedge}\cong \Z/q\Z.\]
Note that in Lemma \ref{Lem 4.4}, we choose a generator $\tfrac{1}{\xi}\in\tfrac{\Qp}{\Zp}[q]$. So the canonical generator of $Q^{\wedge}$ maps to $-r\widetilde{v}D$.
\end{proof}

\begin{defn}\label{Def 4.3}
In the sequence (\ref{Generaljacobian1}), pulling back by the map $\Z/q\Z\rightarrow J[q]$ which takes $1$ to $-r\widetilde{v}D$, and pushing out by the map $T[q]\rightarrow \mu_{q}$ corresponding to $e$, we get an extension class:
\begin{equation}\label{Generaljacobian2}
    0\rightarrow \mu_{q}\rightarrow\,\,?\rightarrow \Z/q\Z\rightarrow 0.
\end{equation}
\end{defn}

\begin{remark}
The extension class (\ref{cuspidal3}) is obtained from the sequence (\ref{torsionses}) via pullback by $e$, and pushout by the map $H^1(X)\rightarrow Q$. If we take the Pontryagin dual, it is the sequence that is obtained from the sequence (\ref{torsionses}) via pullback by the map of $\Z/q\Z\rightarrow J[q]$ which maps $1$ to $-r\widetilde{v}D$ and pushout by $e$. So the extension class (\ref{Generaljacobian2}) is negative of the extension class given by (\ref{cuspidal3}). We will compute the extension class given by (\ref{Generaljacobian2}).
\end{remark}

 For the computation of the extension class given by (\ref{Generaljacobian2}), we have the following general proposition which is from \cite[Section 9.6]{FK}:
\begin{prop}\label{Prop 4.5}
Let $C$ be a smooth projective curve over a field $k$ with characteristic $0$. Let $\Sigma$ be a finite set of $k$-rational points of $C$. Fix two degree $0$ divisors $D_{1}$ and $D_{2}$ such that both of them are supported on $\Sigma$, and $D_{1}$ is order $q$. Starting from the following exact sequence
$$0\rightarrow T[q](\bar{k})\rightarrow GJ[q](\bar{k})\rightarrow J[q](\bar{k})\rightarrow 0,$$
we may pull back by the map $\Z/q\Z\rightarrow J[q]; 1\mapsto D_{1}$ and push out by the map $D_{2}: T[q]\rightarrow \mu_{q}$ to obtain an exact sequence
\begin{equation}\label{Generalses}
0\rightarrow \mu_{q}\rightarrow\,\,?\rightarrow \Z/q\Z\rightarrow 0.
\end{equation}
Let $F$ be a rational function on $C$ such that $qD_{1}=(F)$, and let $h$ be a rational function on $C$ such that $\textnormal{div}(h)-D_{1}$ is supported away from $\Sigma$. Then the extension class of (\ref{Generalses}) coincides with the Kummer class of 
$$(\frac{F}{h^q})(-D_{2}):=\prod_{x\in\Sigma}\frac{F}{h^q}(x)^{-m(x)},$$
where $D_{2}=\sum_{x\in\Sigma} m(x)\cdot x$
\end{prop}

\begin{proof}
We will use the following description of $GJ(\bar{k})$:
\[GJ(\bar{k})=\Div^0_{\Sigma}(\bar{C})/\{\text{div}(f): f\in K(\bar{C}), f\equiv 1 \bmod \Sigma\},\]
where $\Div^0_{\Sigma}(\bar{C})$ is the group of degree $0$ divisors away from $\Sigma$.
Choose a rational function $\alpha\in K(\bar{C})$ such that $\frac{F}{h^q\alpha^q}\equiv 1 \bmod \Sigma$, and let $A=\text{div}(\alpha)$. From the assumption, we know that
\[qD_{1}-q\text{div}(h)-qA=0\in GJ(\bar{k}).\]
So we have $D_{1}-\text{div}(h)-A\in GJ[q](\bar{k})$. Note that 
$D_{1}-\text{div}(h)-A=D_{1} \in J[q](\bar{k})$, since $\text{div}(h)$ and $A$ are principal divisors. Hence the extension class is given by 
\[\sigma\mapsto \sigma(D_{1}-\text{div}(h)-A)-(D_{1}-\text{div}(h)-A).\]
Note that $D_{1}-\text{div}(h)$ is $k$-rational, so the extension class is the class of
\[\sigma\mapsto A-\sigma A,\]
which is the Kummer class of $\alpha^{-q}$. Since $\frac{F}{h^q\alpha^q}\equiv 1 \bmod \Sigma$, we have 
\[\frac{F}{h^q}(x)=\alpha^q(x)\,\,\textnormal{for all}\,\,x\in \Sigma.\]
Hence, after pushout by $D_{2}$, the extension class is given by the Kummer class of $(\frac{F}{h^q})(-D_{2})$.
\end{proof}

\begin{thm}\label{Cor 4.2}
The extension class $c$ is the Kummer class of $N$, and the invariant $\tilde{c}$ equals $1$.
\end{thm}

\begin{proof}
We apply Proposition \ref{Prop 4.5} to the case that the curve is the modular curve $X_{0}(N)$, and $D_{1}=-r\widetilde{v}D=-r\widetilde{v}v((0)-(\infty))$ and $D_{2}=(\infty)-(0)$. We need to find a function $F$ such that
$\text{div}(F)=-qr\widetilde{v}D$. Let $\Delta(z)\in S_{12}(\SL_{2}(\Z),\C)$ be the discriminant function and let $g=(\frac{\Delta(Nz)}{\Delta(z)})^{\tfrac{1}{12}}$. we know that $\text{div}(g^{-1})=\xi((0)-(\infty))$. Hence $\text{div}(g^{\frac{m}{v}r\widetilde{v}v})=-qr\widetilde{v}v((0)-(\infty))$. It is easy to see that $ \frac{m}{v}r\widetilde{v}v\equiv 1\bmod q$. By computation, we have 
\[\frac{g^{\frac{m}{v}r\widetilde{v}v}(0)}{g^{\frac{m}{v}r\widetilde{v}v}(\infty)}=N^{a},\] 
where $a\equiv -1\bmod q$, which means that the extension class (\ref{Generaljacobian2}) is given by the Kummer class of $N^{-1}$. Therefore, the extension class (\ref{cuspidal3}) is given by the Kummer class of $N$.
By the isomorphism of Lemma \ref{Lem 4.1} (2), the invariant $\tilde{c}$ equals $1$.
\end{proof}

\section{The cocycle $\chi_{b}$}

Recall the cocycle $\chi_{b}\in \Hom(G_{-1},P\otimes I/I^2)\cong \Hom(G_{-1},I^2/I^3\otimes \mu_{q}^{-1})$. In this section, we try to give a conjectural inverse map to $\chi_{b}$.

\subsection{\textit{K}-theory of integer rings and Galois cohomology}
We will use the following theorem to identify $K$-groups with \'etale cohomology groups:

\begin{thm}\label{Thm 4.1}
For a number field $F$, let $S$ be a finite set of places containing the places above $p$. Then there is a \'etale Chern character inducing an isomorhism:
\[K_{2}(\mathfrak{o}_{F,S})\otimes\Zp\rightarrow H^2_{\textnormal{\'et}}(\mathfrak{o}_{F,S},\Zp(2)).
\]
\end{thm}

For the proof, one can see \cite{Sou}.

\begin{remark}
We may identify Galois cohomology groups with \'etale cohomology groups: for the details, see \cite{Ma0}.
\end{remark}

\begin{lemma}\label{new1}
We have the following exact sequence:
\[0\rightarrow K_{2}(\Z[\zeta_{N}])\otimes\Zp\rightarrow K_{2}(\Z[\zeta_{N},\tfrac{1}{N}])\otimes\Zp\rightarrow K_1(\F_{N})\otimes\Zp\rightarrow 0.\]
\end{lemma}

\begin{proof}
This is from the localization sequence of $K$-groups and the facts that
\begin{enumerate}
    \item $K_{2}(\F_{N})=0$,
    \item $K_{1}(\Z[\zeta_{N}])\otimes\Zp\hookrightarrow K_1(\Z[\zeta_{N},\tfrac{1}{N}])\otimes\Zp$.
\end{enumerate}
\end{proof}

\begin{cor}\label{new2}
Using the Chern class maps of Theorem \ref{Thm 4.1}, we have the following exact sequence in \'etale cohomology:
\[0\rightarrow H^{2}(\Z[\zeta_{N},\tfrac{1}{p}],\Zp(2))\rightarrow H^{2}(\Z[\zeta_{N},\tfrac{1}{Np}],\Zp(2))\rightarrow \F_{N}^{\times}\otimes\Zp\rightarrow 0.\]
\end{cor}

Let $G=\Gal(\Q(\zeta_{N})/\Q)$. We have the following lemma:

\begin{lemma}\label{new3}
Via the injection of Corollary \ref{new2}, we have
\[H^{2}(\Z[\zeta_{N},\tfrac{1}{p}],\Zp(2))=I_{G}H^{2}(\Z[\zeta_{N},\tfrac{1}{Np}],\Zp(2)).\]
\end{lemma}

\begin{proof}
It suffices to show that the map induced on coinvariants
\[H^{2}(\Z[\zeta_{N},\tfrac{1}{Np}],\Zp(2))_{G}\rightarrow \F_{N}^{\times}\otimes\Zp\]
from the exact sequence of Corollary \ref{new2} is an isomorphism. But \[H^{2}(\Z[\zeta_{N},\tfrac{1}{Np}],\Zp(2))_{G}\cong H^2(\Z[\tfrac{1}{Np}],\Zp(2))\]
by \cite[Proposition 3.3.11]{NSW} and \[H^2(\Z[\tfrac{1}{Np}],\Zp(2))\cong K_{2}(\Z[\tfrac{1}{Np}])\otimes\Zp=\F_{N}^{\times}\otimes\Zp,\] 
so the surjection is an isomorphism by equality of orders.
\end{proof}

\begin{cor}\label{new4}
Fix the isomorphism
\[I_{G}/I_{G}^2\rightarrow G,\]
\[g-1\mapsto g.\]
We have canonical isomorphisms:
\[H^2(\Z[\zeta_{N},\tfrac{1}{p}],\Zp(2))_{G}\cong \frac{I_{G}H^{2}(\Z[\zeta_{N},\tfrac{1}{Np}],\Zp(2))}{I^2_{G}H^{2}(\Z[\zeta_{N},\tfrac{1}{Np}],\Zp(2))}\cong I_{G}/I_{G}^2\otimes\F_{N}^{\times}\otimes\Zp\cong U^{\otimes 2}.\]
\end{cor}

\begin{prop}\label{Prop 4.6}
We have a canonical isomorphism:
$$H^2(\Z[\zeta_{N},\tfrac{1}{p}],\Zp(2))_{G}\cong G_{-1}\otimes\mu_{q}.$$
\end{prop}

\begin{proof}
Let $\Delta=\Gal(\Q(\zeta_{q})/\Q)$. 
Since $H^2(\Z[\zeta_{N},\frac{1}{p}],\Zp(2))_{G}\cong U^{\otimes 2}\cong \Z/q\Z$ and $\Spec\Z[\zeta_{N},\tfrac{1}{p}]$ has $p$-cohomological dimension $2$ (see \cite[Proposition 8.3.18]{NSW}), we know that
\[
H^2(\Z[\zeta_{N},\tfrac{1}{p}],\Zp(2))_{G}\cong H^2(\Z[\zeta_{N},\tfrac{1}{p}],\mu_{q}^{\otimes 2})_{G}.
\]
Since
\[H^2(\Z[\zeta_{N},\tfrac{1}{p}],\mu_{q}^{\otimes 2})\cong H^2(\Z[\zeta_{Nq},\tfrac{1}{p}],\mu_{q}^{\otimes 2})_{\Delta}\]
by \cite[Proposition 3.3.11]{NSW},
we only need to understand  $H^2(\Z[\zeta_{Nq},\frac{1}{p}],\mu_{q}^{\otimes 2})_{G\times \Delta}$. By \cite[Proposition 8.3.11]{NSW}, we have the following sequence
\[0\rightarrow A_{\Q(\zeta_{Nq}),p}\otimes\Z/q\Z\rightarrow  H^2(\Z[\tfrac{1}{p},\zeta_{Nq}],\mu_{q})\rightarrow \bigoplus_{v\mid p} H^2(\Q(\zeta_{Nq})_{v},\mu_{q})\rightarrow \Z/q\Z\rightarrow 0,\]
where $A_{\Q(\zeta_{Nq}),p}$ is ideal class group of $\Q(\zeta_{Nq})$ modulo the ideal classes $[\p]$, for $\p| p$. Taking the $\kappa^{-1}$-eigenquotient, where $\kappa$ is the modulo $q$ cyclotomic character, since $H^2(\Q(\zeta_{Nq})_{v},\mu_{q}^{\otimes 2})\cong \mu_{q}$ and $\Delta$ does not permute the $v\mid p$, we have $(\bigoplus_{v\mid p} H^2(\Q(\zeta_{Nq})_{v},\mu_{q}^{\otimes 2})_{\Delta}=0$. Thus, we have the following:
\[(A_{\Q(\zeta_{Nq}),p})_{\kappa^{-1}}\otimes \mu_{q}\cong H^2(\Z[\zeta_{Nq},\tfrac{1}{p}],\mu_{q}^{\otimes 2})_{\Delta}.\]
Hence 
\[H^2(\Z[\zeta_{Nq},\tfrac{1}{p}],\mu_{q}^{\otimes 2})_{G\times \Delta}\cong(A_{\Q(\zeta_{Nq}),p})_{\kappa^{-1},G}\otimes \mu_{q}.
\]
Now let $F/\Q(\zeta_{Nq})$ be the field extension that corresponds to $(A_{\Q(\zeta_{Np}),p})_{\kappa^{-1}}$.
Note that by Proposition \ref{Prop 4.1}, we have
\[G_{-1}\cong \Gal(F_{-1}/K),\]
where $F_{-1}$ is tamely ramified at $N$ and split completely at the prime dividing $p$. Hence, $F_{-1}\Q(\zeta_{Nq})$ is unramfied over $\Q(\zeta_{Nq})$ with $\kappa^{-1}$-action and split completely at the primes dividing $p$.  Then, we have  $F_{-1}\Q(\zeta_{Nq})\subset F$, which means that 
$G_{-1}$ is a quotient of $(A_{\Q(\zeta_{Nq}),p})_{\kappa^{-1},G}$, but they have the same size, so we have
\[
(A_{\Q(\zeta_{Nq}),p})_{\kappa^{-1},G}\cong G_{-1}.
\]

\end{proof}

\begin{remark}\label{canonicaliso}
By Corollary \ref{new4} and Proposition \ref{Prop 4.6}, we have a canonical isomorphism:
\[G_{-1}\cong U^{\otimes 2}\otimes\mu_{q}^{-1}.\]
\end{remark}

\subsection{The map from homology to Galois cohomology}
\begin{defn}[Adjusted Manin symbol]\label{Def 4.5}
For $u,v\in \Z/N\Z$ with $(u,v)=1$, we let 
$$[u,v]'=W_{N}[u,v]=\Big\{\frac{-d}{bN},\frac{-c}{aN}\Big\},$$
where $a,b,c,d\in\Z$ satisfy $ad-bc=1,u=c\mod N, v=d\mod N$.

Let $[u,v]^{*}:=W_{N}[u,v]^+=\frac{1}{2}([u,v]'+[u,-v]')$
\end{defn}

\begin{lemma}\label{Lem 4.5}
The diamond operator $\langle j\rangle\,(j\in\F_{N}^{\times})$ acts on $[u,v]'$ as follows:
$$\langle j\rangle [u,v]'=[j^{-1}u,j^{-1}v]'.$$
\end{lemma}

\begin{proof}
We have
\[\langle j\rangle [u,v]'=\langle j\rangle W_{N}[u,v]=W_{N}\langle j\rangle^{*}[u,v]=[j^{-1}u,j^{-1}v]'.\]
\end{proof}

\begin{thm}[Manin]\label{Thm 4.2}
The relative homology group $H_1(X_{1}(N),\textnormal{cusps},\Zp)^+$ has a representation as a $\Zp[\F_{N}^{\times}]$-module with generators $[u,v]^{*}$ for $u,v\in \F_{N}$ and $(u,v)=(1)$, subject to the following relations:
\begin{enumerate}
    \item $[u,v]^{*}+[-v,u]^{*}=0,$
    \item $[u,v]^{*}=[u,u+v]^{*}+[u+v,v]^{*},$
    \item $[-u,-v]^{*}=[u,v]^{*}=[u,-v]^{*},$
    \item $\langle j\rangle [u,v]^{*}=[j^{-1}u,j^{-1}v]^{*}$ for $j\in\F_{N}^{\times}$.
\end{enumerate}
\end{thm}

\begin{remark}\label{Rem 4.2}
We have the following isomorphisms via Poincar\'e duality
\[
H^1(X_{1}(N)(\C),\Zp)\cong H_1(X_{1}(N)(\C),\Zp),
\]
\[H^1(Y_{1}(N)(\C),\Zp)\cong H_1(X_{1}(N)(\C),\textnormal{cusps},\Zp).
\]
Note that these isomorphisms are not Hecke compatible, they transfer from the adjoint $T(l)^{*}$- action to the $T(l)$-action. For the details, one can see \cite[Proposition 3.5]{Sha1}.
\end{remark}

\begin{remark}\label{Rem 4.3}
If we identify $H^1(X_{1}(N)(\C),\Zp)$ (resp. $H^1(Y_{1}(N)(\C),\Zp)$) with $H_{\textnormal{\'et}}^1(X_{1}(N))$ (resp. $H_{\textnormal{\'et}}^1(Y_{1}(N))$), Poincar\'e duality also changes the Galois action. In fact we have an isomorphism
\[
H_{\textnormal{\'et}}^1(X_{1}(N))(1)\cong H_1(X_{1}(N)(C),\Zp),
\]
which respects complex conjugation. For the details, one can see \cite[Section 3.5]{Sha1}.
\end{remark}

\begin{lemma}\label{Lem 4.6}
Let $\tilde{C_{0}}$ denote the cusps of $X_{1}(N)(\C)$ that lie above the $0$ cusp in $X_{0}(N)$, and let $\tilde{C}_{\infty}$ denote the cusps of $X_{1}(N)(\C)$ that lie above the $\infty$ cusp in $X_{0}(N)$. The
Manin symbols $[u,v]\,\,(u\ne 0, v\ne 0)$ generate the relative homology group $H_1(X_{1}(N),\tilde{C_{0}},\Zp)$, and Manin symbols $[u,v]'$ generate the relative homology group $H_1(X_{1}(N),\tilde{C}_{\infty},\Zp)$. 
\end{lemma}

\begin{prop}[Sharifi]\label{Prop 4.8}
There exists a homomorphism
\[\varpi_{N}^{0}:H_{1}(X_{1}(N),\tilde{C}_{\infty},\Zp)^{+}\rightarrow H^2(\Z[\zeta_{N},\tfrac{1}{Np}],\Zp(2))^{+}.\]
\[[u,v]^{*}\mapsto (1-\zeta_{N}^u,1-\zeta_{N}^v)^{+}.
\]
\end{prop}

\begin{defn}\label{Def 4.6}
Since $H_{1}(X_{1}(N),\Zp)^{+}\subset H_{1}(X_{1}(N),\tilde{C}_{\infty},\Zp)^{+}$, via restriction, we have the following homomorphism:
$$\varpi_{N}:H_{1}(X_{1}(N),\Zp)^{+}\rightarrow H^2(\Z[\zeta_{N},\tfrac{1}{Np}],\Zp(2))^{+}.$$
\end{defn}

\begin{defn}(Eisenstein ideal for $X_{1}(N)$)\label{Def 4.7}
Let $\mathfrak{I}_{0}$ be the ideal of $\End_{\Zp}(H_1(X_{1}(N),\tilde{C}_{\infty},\Zp))$ generated by ${T(l)-l-\langle l\rangle,\,\,\textnormal{for}\,\, l\neq N,\,\,\textnormal{and}\,\,T(N)-N }$, and let $I_{0}$ be the image of $\mathfrak{I}_{0}$ in $\End_{\Zp}(H_1(X_{1}(N),\Zp))$.
Let $\mathfrak{I}_{\infty}$ be the ideal of $\End_{\Zp}(H_1(X_{1}(N),\tilde{C}_{\infty},\Zp))$ generated by ${T(l)-1-l\langle l\rangle,\,\,\textnormal{for}\,\,l\neq N,\,\,\textnormal{and}\,\,T(N)-1 }$, and let $I_{\infty}$ be the image of $\mathfrak{I}_{\infty}$ in $\End_{\Zp}(H_1(X_{1}(N),\Zp))$.
\end{defn}

\begin{lemma}\label{Lem 4.7}
We have $\mathfrak{I}_{0} H_1(X_{1}(N),\tilde{C}_{\infty},\Zp)\subset H_1(X_{1}(N),\Zp)$.
\end{lemma}
\begin{proof}
One can check that $\mathfrak{I}_{0}$ kills $\tilde{C}_{\infty}$. For the Hecke action on cusps, one can see \cite[Section 1.2]{Oh2}.
\end{proof}

\subsection{Topological boundary and arithmetic boundary}
In this section, follow the ideal in \cite[Secion 5.3]{FK}, we compare the following three exact sequences:
\begin{equation}\label{Boundary1}
  0\rightarrow H^2(\Z[\tfrac{1}{p},\zeta_{N}^+],\Zp(2))\rightarrow H^2(\Z[\tfrac{1}{Np},\zeta_{N}^+],\Zp(2))\xrightarrow{\partial} H^2(\Q(\zeta_{N}^+)_{N},\Zp(2))\rightarrow 0,
\end{equation}
where $\Q(\zeta_{N}^+)_{N}$ is the completion of $\Q(\zeta_{N}^+)$ at the place dividing $N$,
\begin{equation}\label{Boundary2}
0\rightarrow K_{2}(\Z[\zeta_{N}^+])\otimes\Zp\rightarrow K_{2}(\Z[\zeta_{N}^+,\tfrac{1}{N}])\otimes\Zp\xrightarrow{\partial}K_{1}(\F_{N})\otimes \Zp\rightarrow 0,
\end{equation}
\begin{equation}\label{Boundary3}
    0\rightarrow  H_{1}(X_{1}(N),\Zp)^+\rightarrow H_{1}(X_{1}(N),\tilde{C}_{\infty},\Zp)^+\xrightarrow{\partial'}\sideset{^{}_{}}{^0_{}}
\bigoplus_{\widetilde{C}_{\infty}}\Zp\rightarrow 0,
\end{equation}
where $\bigoplus^0$ denotes the elements which sum to zero. 

\begin{remark}\label{Rem 4.4}
We identify sequence (\ref{Boundary1}) with the sequence (\ref{Boundary2}) via the \'etale chern class map.
\end{remark}

\begin{prop}\label{Prop 4.9}
In sequence (\ref{Boundary1}), the residue map $\partial$ satisfies
\[
\partial(1-\zeta_{N}^u,1-\zeta_{N}^v)^+=\frac{u}{v}\in\F_{N}^{\times}\otimes\Zp.
\]
\end{prop}

\begin{proof}
By the definition of tame symbol, we have
\[\partial(1-\zeta_{N}^u,1-\zeta_{N}^v)=(-1)\frac{1-\zeta_{N}^u}{1-\zeta_{N}^v}\equiv -\frac{u}{v}\in\F_{N}^{\times}\otimes\Zp.\]
Note that $-1$ is trivial in $\F_{N}^{\times}\otimes\Zp$.
\end{proof}

\begin{prop}\label{Prop 4.10}
In sequence (\ref{Boundary3}), the boundary map $\partial'$ satisfies
\[\partial'([u,v]^{*})=\Big(\frac{-c}{Na}\Big)-\Big(\frac{-d}{Nb}\Big),
\]
where $a,b,c,d\in\Z$ with $ad-bc=1$, and $(u,v)\equiv (c,d)\bmod N$.
\end{prop}

\begin{proof}
It is almost immediate from the definition.
\end{proof}

\begin{defn}\label{Def 4.8}
Let $s=\tfrac{c}{a}$ be an element of $\Q\cup \{\infty\}$ with $\gcd(a,c)=1$ and $N\mid a$. We define a map $t$ as follows:
\[
t: \Zp[\widetilde{C}_{\infty}]\rightarrow \F_{N}^{\times}\otimes\Zp,
\]
\[
t(s)=c\in \F_{N}^{\times}\otimes\Zp.
\]
\end{defn}

\begin{remark}
The map $t$ is well defined. One can check this by using \cite[Proposition 3.8.3]{DS}.
\end{remark}

\begin{prop}
The following diagram commutes:
\begin{equation}\label{tamediagram}
\begin{tikzcd}
H_1(X_{1}(N),\tilde{C}_{\infty},\Zp)\ar[r,"\partial"]\ar[d,"\varpi^{0}"]& \underset{\widetilde{C}_{\infty}}{\bigoplus^0}\,\,\Zp\ar[d,"t"]\\
K_{2}(\Z[\zeta_{N}^+, \frac{1}{N}])\otimes\Zp\ar[r,"\partial'"]& \F_{N}^{\times}\otimes\Zp.
\end{tikzcd}
\end{equation}
\end{prop}

\begin{proof}
Using Proposition \ref{Prop 4.10}, we know that
\[t(\partial([u,v]^{*}))=\frac{u}{v}.\]
And using Proposition \ref{Prop 4.9} and the definition of $\varpi_{N}^0$, we know that
\[\partial'\varpi_{N}^{0}([u,v]^{*})=\frac{c}{d}= \frac{u}{v}\in\F_{N}^{\times}\otimes\Zp.\]
\end{proof}

So we have the following proposition.
\begin{prop}\label{Prop 4.11}
The image of $\varpi_{N}$ is in $H^2(\Z[\zeta_{N},\frac{1}{p}],\Zp(2))^{+}$ or $K_{2}(\Z[\zeta_{N}^+])\otimes\Zp$.
\end{prop}
\begin{proof}
This is the map on kernels of the horizontal maps in the commutative diagram (\ref{tamediagram}).
\end{proof}

\begin{remark}\label{Rem 4.5}
By Remark \ref{Rem 4.3}, we may also view $\varpi$ as a map:
$$\varpi_{N}: H^1(X_{1}(N),\Zp)^{-}(1)\rightarrow H^2(\Z[\tfrac{1}{p},\zeta_{N}],\Zp(2))^{+}.$$
\end{remark}

\subsection{Sharifi's conjecture}
In this section, we list several conjectures related to Sharifi's conjectures.
\begin{conj}[Sharifi]\label{Con 4.1}
The map $\varpi_{N}^{0}$ satisfies 
\[
\varpi_{N}^0(\eta x)=0
\]
for all $\eta\in \mathfrak{I}_{\infty}$ and $x\in H_1(X_{1}(N),\tilde{C}_{\infty},\Zp)$.
\end{conj}

\begin{lemma}[Busuioc, Sharifi]\label{Lem 4.8}
We have
\[\varpi^0\circ(T(2)-1-2\langle 2\rangle)=0,
\]
\[\varpi^0\circ(T(3)-1-3\langle 3\rangle)=0.
\]
\end{lemma}

\begin{proof}
For the proof, one can see \cite[Theorem 1.2]{BA}.
\end{proof}

\begin{remark}\label{Rem 4.6}
Conjecture \ref{Con 4.1} has been proved when the level of the modular curve is $Np^r$ ($r>0$): see \cite[Theorem 5.2.3]{FK}. Similar conjecture has also been considered in \cite{Lec2}. We will give some partial results for this conjecture in the forthcoming paper.
\end{remark}

Let $\widetilde{I}=I_{\infty}+I_{G}$. Then we have the following map of exact sequences which gives the relationship between the cohomology of $X_{1}(N)$ and $X_{0}(N)$:
\begin{equation*}
\begin{tikzcd}
0\arrow[r]& \frac{H^1(X_{1}(N),\Zp)^{-}}{\widetilde{I}H^1(X_{1}(N),\Zp)^{-}}\arrow[d,"\pi",]\arrow[r]& \frac{H^1(X_{1}(N),\Zp)}{\widetilde{I}H^1(X_{1}(N),\Zp)}\arrow[d,"\cong"]\arrow[r]& \frac{H^1(X_{1}(N),\Zp)^{+}}{\widetilde{I}H^1(X_{1}(N),\Zp)^{+}}\arrow[d,"\cong"]\arrow[r]&0\\
0\arrow[r]&\frac{IH^1(X_{0}(N),\Zp)^{-}}{I^2H^1(X_{0}(N),\Zp)^{-}}\arrow[r]&\frac{IH^1(X_{0}(N))^{-}\oplus H^1(X_{0}(N))^{+}}{I^2H^1(X_{0}(N))^{-}\oplus IH^1(X_{0}(N))^{+}}\arrow[r]&\frac{H^1(X_{0}(N),\Zp)^{+}}{IH^1(X_{0}(N),\Zp)^{+}}\arrow[r]&0
\end{tikzcd}
\end{equation*}

\begin{remark}
The surjectivity of the vertical maps is from Theorem \ref{Thm 3.3}. The injectivity of the vertical maps is from \cite[Proposition 4.5]{Lec2}.
\end{remark}
\begin{remark}\label{Rem 4.7}
Note that the sequence  
\[
0\rightarrow \frac{H^1(X_{1}(N),\Zp)^{-}}{\widetilde{I}H^1(X_{1}(N),\Zp)^{-}}\rightarrow \frac{H^1(X_{1}(N),\Zp)}{\widetilde{I}H^1(X_{1}(N),\Zp)}\rightarrow \frac{H^1(X_{1}(N),\Zp)^{+}}{\widetilde{I}H^1(X_{1}(N),\Zp)^{+}}\rightarrow 0
\]
also gives the same extension class as $b$ in $H^1(\Z[\frac{1}{Np}],P\otimes I/I^2)$. It gives a character 
\begin{equation*}
\chi_{b}:G_{-1}\rightarrow \frac{H^1(X_{1}(N),\Zp)^{-}}{\tilde{I}H^1(X_{1}(N),\Zp)^{-}}\cong \frac{IH^1(X_{0}(N),\Zp)^{-}}{I^2H^1(X_{0}(N),\Zp)^{-}}\cong I/I^2\otimes P.
\end{equation*}
\end{remark}

Now, we make the following conjecture that is analogous to Conjecture \ref{Sharificonj} in the Introduction.

\begin{conj}\label{Con 4.2}
\begin{enumerate}
\item The map $\varpi_{N}$ induces an isomorphism:
\[
\frac{H^1(X_{1}(N),\Zp)^-(1)}{I_{\infty}H^1(X_{1}(N),\Zp)^-(1)}\cong\frac{H_1(X_{1}(N),\Zp)^+}{I_{\infty}H_1(X_{1}(N),\Zp)^+}
\xrightarrow{\varpi} H^2(\Z[\zeta_{N}^+,\tfrac{1}{p}],\Zp(2)).
\]
\item Let $\varpi_{N,G}$ be the map induced by $\varpi_{N}$ on $G$-coinvariants. It is a map
\[\frac{IH^{1}(X_{0}(N),\Zp)^{-}(1)}{I^2H_{1}(X_{0}(N),\Zp)^{-}(1)}\rightarrow H^2(\Z[\zeta_{N},\tfrac{1}{p}],\Zp(2))_{G}\cong G_{-1}\otimes\mu_{q}.\]
Twisting the coefficients, we have a map
\[\varpi_{N,G}\otimes\mu_{q}^{-1}:\frac{IH^{1}(X_{0}(N),\Zp)^{-}}{I^2H_{1}(X_{0}(N),\Zp)^{-}}\rightarrow H^2(\Z[\zeta_{N}^+,\tfrac{1}{p}],\Zp(2))_{G}\otimes \mu_{q}^{-1}\cong G_{-1}.\]
Then
\[\chi_{b}\circ (\varpi_{N,G}\otimes\mu_{q}^{-1})=(\varpi_{N,G}\otimes\mu_{q}^{-1})\circ \chi_{b}=1.\]
\end{enumerate}
\end{conj}

\begin{lemma}\label{SB1}
In the following diagram
\[\begin{tikzcd}
H_1(X_{1}(N),\tilde{C}_{\infty},\Zp)^+\ar[d,"\pi"]\ar[r,"\varpi_{N}^0"]& H^2(\Z[\zeta_{N}^+,\tfrac{1}{Np}],\Zp(2))\ar[d]\\
H_1(X_{0}(N),\Zp)^+\ar[r,"\varpi_{N,G}^0"]& H^2(\Z[\tfrac{1}{Np}],\Zp(2)),
\end{tikzcd}\]
the vertical maps are both surjective. The map $\varpi_{N,G}^0$ factors through the quotient by $IH_1(X_{0}(N),\Zp)^+$.
\end{lemma}

\begin{proof}
The surjectivity of vertical maps is from the definitions and the fact that there are only two cusps on $X_{0}(N)$. Note that $\pi([u,v]^{*})=-\pi(\langle v\rangle[\frac{u}{v},1])=-[\frac{u}{v},1]=-\{0,\frac{1}{\frac{v}{u}}\}$. By Proposition \ref{Prop 4.9}, $[u,v]^{*}$ maps to $\frac{u}{v}\in H^2(\tfrac{1}{Np},\Zp(2))$. So $\varpi_{N,G}^0$ maps $\{0,\frac{1}{\frac{v}{u}}\}$ to $\frac{v}{u}$. From this, we know that $\varpi_{N,G}^0=\phi^{-1}$, where $\phi$ is the map defined in Theorem \ref{mazur}.  By Theorem \ref{mazur}, we know that $\varpi_{N,G}^0$ factors through $IH_1(X_{0}(N),\Zp)^+$.
\end{proof}

\begin{remark}
We also have the following commutative diagram
\[\begin{tikzcd}
H_1(X_{1}(N),\Zp)^+\ar[d,"\pi"]\ar[r,"\varpi_{N}"]& H^2(\Z[\zeta_{N}^+,\tfrac{1}{p}],\Zp(2))\ar[d]\\
IH_1(X_{0}(N),\Zp)^+\ar[r,"\varpi_{N,G}"]& H^2(\Z[\zeta_{N}^+,\tfrac{1}{p}],\Zp(2))_{G}.
\end{tikzcd}\]
However, we cannot prove that $\varpi_{N,G}$ factors through the quotient by $I^2H_1(X_{0}(N),\Zp)^+$.
\end{remark}

\subsection{Computation of \textit{b}}
In this section, assuming Conjecture \ref{Con 4.1} and Conjecture \ref{Con 4.2}, we compute the invariant $\tilde{b}$. 
Recall that in Theorem \ref{mazur}, we have a canonical isomorphism:
\[\phi:U\rightarrow H_{1}(X_{0}(N),\Zp)^+/IH_{1}(X_{0}(N),\Zp)^+\]
We identify the two groups by $\phi$.

\begin{lemma}\label{Lem 4.9}
For $u,v\in\F_{N}^{\times}$,
the image of $(l+\langle l\rangle-T(l))[u,v]^{*}$ in $\frac{IH_1(X_{0}(N),\Zp)^{+}}{I^2H_{1}(X_{0}(N),\Zp)^{+}}\cong I/I^2\otimes \frac{H_1(X_{0}(N),\Zp)^{+}}{IH_{1}(X_{0}(N),\Zp)^{+}}$ is $\eta_{l}\otimes\frac{u}{v}$, where we view $\frac{u}{v}$ as an element in $U\cong \frac{H_1(X_{0}(N),\Zp)^{+}}{IH_{1}(X_{0}(N),\Zp)^{+}}$.
\end{lemma}

\begin{proof}
By definition, we have $\pi((l+\langle l\rangle-T(l))[u,v]^{*})=\eta_{l}\pi([u,v]^{*})$. Now, let us compute $\pi([u,v]^{*})=\pi(W_{\zeta_{N}}([u,v]))=W_{N}\pi([u,v])$. Using Hensel's lemma, we have that the image of $W_{N}$ in $\mathfrak{h}_{0}(N)_{\mathfrak{P}}$. Hence, in $H_1(X_{0}(N),\Zp)^+_{\mathfrak{P}}$, we have $\pi([u,v]^{*})=-\pi(\langle v\rangle[\frac{u}{v},1])=-[\frac{u}{v},1]$. Let $x\in \Z$ be a lifting of $\frac{u}{v}$. By computation, we have $[\frac{u}{v},1]=\{0,\frac{1}{x}\}$. Using Theorem \ref{mazur}, we have that the element $-\{0,\frac{1}{x}\}$ corresponds to $\frac{u}{v}\in U$.
\end{proof}

\begin{lemma}\label{Lem 4.10}
Via $\partial\circ\varpi_{N}^0$, $[u,v]^{*}$ maps to $\frac{u}{v}$ in $H^2(\Q_{N}(\zeta_{N}),\Zp(2))\cong U$.
\end{lemma}

\begin{proof}
The image of $\varpi_{N}^0([u,v]^{*})$ in $H^2(\Q_{N}(\zeta_{N}),\Zp(2))$ is $(1-\zeta_{N}^u,1-\zeta_{N}^v)$ by definition. Via the map $H^2(\Q_{N}(\zeta_{N}),\Zp(2))\cong U$, the image of $(1-\zeta_{N}^u,1-\zeta_{N}^v)$ is exactly $\frac{u}{v}$ by Proposition \ref{Prop 4.9}. 
\end{proof}

Since $(l+\langle l\rangle-T(l))[u,v]^{*}\in H_1(X_{1}(N),\Zp)$, via $\varpi_{N}$, it maps into 
$$I_{G}/I_{G}^2 \otimes H^2(\Z[\tfrac{1}{Np},\zeta_{N}],\Zp(2))\cong G\otimes U.$$

\begin{lemma}\label{Lem 4.11}
Via $\varpi_{N}$, the element $(l+\langle l\rangle-T(l))[u,v]^{*}$ maps to $l^{l-1}\otimes \frac{u}{v}$.
\end{lemma}
\begin{proof}
Note that by definition, we have 
\[\varpi_{N}((l+\langle l\rangle-T(l))[u,v]^{*})=\varpi_{N}^0((l+\langle l\rangle-T(l))[u,v]^{*}).\]
If we assume the Eisenstein quotient conjecture for $\varpi_{N}^0$ and note that 

\begin{equation}
    l+\langle l\rangle -T(l)=(l-1)(1-\langle l\rangle )\mod I_{\infty}.
\end{equation}
We have 
\[\varpi_{N}^0((l+\langle l\rangle-T(l))[u,v]^{*})=\varpi_{N}^0((l-1)(1-\langle l\rangle)[u,v]^{*})=(l-1)(1-\sigma_{l}^{-1})\varpi_{N}^0([u,v]^{*}).\]

\npr By the isomorphism in Corollary \ref{new4}, $(l-1)(1-\sigma_{l}^{-1})$ maps to $l^{l-1}\in G\otimes\Zp$. Since $\partial\varpi_{N}^0([u,v]^{*})=\frac{u}{v}\in U$, we have proved that $\varpi_{N}((l+\langle l\rangle-T(l))[u,v]^{*})=l^{l-1}\otimes \frac{u}{v}$.
\end{proof}

\begin{thm}\label{computationb}
Suppose that Conjecture \ref{Con 4.1} and Conjecture \ref{Con 4.2} hold. Then the invariant $\tilde{b}$ equals 1.
\end{thm}

\begin{proof}
By Lemma \ref{Lem 4.9} and Lemma \ref{Lem 4.11}, we know that
\[\varpi_{N,G}(\eta_{l}\otimes\frac{u}{v})=l^{l-1}\otimes\frac{u}{v}.\]
Assuming Conjecture \ref{Con 4.2}, we know that
\[\chi_{b}(l^{l-1}\otimes\frac{u}{v}\otimes\zeta_{q})=\eta_{l}\otimes\frac{u}{v}\otimes \zeta_{q}.\]
Since the winding isomorphism identifies $l^{l-1}$ with $\eta_{l}$ (Theorem \ref{mazur}), we know that $\tilde{b}=1$.
\end{proof}

\end{document}